\documentclass[12pt]{amsart}%
\usepackage{color}
\usepackage{amsmath}
\usepackage{amssymb}
\usepackage{amsfonts}
\usepackage[latin1]{inputenc}
\usepackage{graphicx}%
\providecommand{\U}[1]{\protect\rule{.1in}{.1in}}

\DeclareMathSymbol{\subsetneqq}{\mathbin}{AMSb}{36}

\theoremstyle{plain}
\numberwithin{equation}{section}
\newtheorem{theorem}{Theorem}[section]

\newtheorem{lemma}{Lemma}[section]
\newtheorem{proposition}{Proposition}[section]
\newtheorem{definition}{Definition}[section]
\newtheorem{remark}{Remark}[section]
\newtheorem{notation}{Notation}
\begin{document}
\title[Extension of an unicity class for Navier-Stokes equations]
{Extension of an unicity class for Navier-Stokes equations}%
\author{Ramzi May}%
\address{Mathematics Department, College of Science of Bizerte, Carthage University, Bizerte, Tunisia}
\email{rmay@kfu.edu.sa}
 \vskip 0.2cm
\vskip 0.2cm
\date{April 9, 2018}
\dedicatory{}
\begin{abstract}
This is a translation from French of my paper [R.
May, Extension d'une classe d'unicite pour les equations de Navier-Stokes,
Ann. I. H. Poincar\'{e}-AN 27 (2010) 705-718. doi:10.1016/j.anihp.2009.11.007].

Q. Chen, C. Miao, and Z. Zhang \cite{CMZ} have proved that weak Leray
solutions of the Navier-Stokes are unique in the class $L^{\frac{2}{1+r}%
}([0,T].B_{\infty}^{r,\infty}(\mathbb{R}^{3})$ with $r\in]-\frac{1}{2},1].$ In
this paper, we establish that this criterion remains true for $r\in
]-1,-\frac{1}{2}].$
\end{abstract}
\maketitle

\section{Introduction and statement of the results}

We consider the Navier-Stokes equations for an incompressible fluid in the
entire space $\mathbb{R}^{d},d\geq2,$%

\begin{equation}
\left\{
\begin{array}
[c]{c}%
\partial_{t}u-\Delta u+\nabla(u\otimes u)+\overrightarrow{\nabla}p=0,\\
\overrightarrow{\nabla}.u=0,\\
u(0,.)=u_{0}(.)
\end{array}
\right.  \tag{NS}%
\end{equation}
where $u_{0}$ is the initial velocity of the fluid particles, $u=u(t,x)$
designs of the particle placed in $x\in\mathbb{R}^{d}$ at the time $t\geq0,$
$p=p(t,x)$ is the pressure at $x\in\mathbb{R}^{d}$ and $t\geq0,$
$\overrightarrow{\nabla}=(\partial_{x_{1}},\cdots,\partial_{x_{d}})^{t}$
denotes the gradient operator, $\overrightarrow{\nabla}.$ is the divergence
operator, and $\nabla(u\otimes u)$ is the vector function $(w_{1},\cdots
,w_{d})$ defined by%
\[
w_{i}=\sum_{k=1}^{d}\partial_{x_{k}}(u_{k}u_{i})=\overrightarrow{\nabla
}.(u_{i}u).
\]

Let us first recall the notion of the weak solutions for the Navier-Stokes
equations that we will adopt in this paper.

\begin{definition}
Let $T\in]0,+\infty]$ and $u_{0}=(u_{01},\cdots,u_{0d})\in(S^{\prime
}(\mathbb{R}^{d}))^{d}$ with divergence free. A weak solution on $]0,T[$ of
the equations (NS) is a function $u:Q_{T}\equiv]0,T[\times\mathbb{R}%
^{d}\rightarrow\mathbb{R}^{d}$ satisfying the following properties:
\begin{enumerate}
\item[1)] $u\in L_{loc}^{2}(\tilde{Q}_{T})$ where $\tilde{Q}_{T}\equiv
\lbrack0,T[\times\mathbb{R}^{d}.$

\item[2)] $u\in C([0,T[,S^{\prime}(\mathbb{R}^{d})).$

\item[3)] $u(0)=u_{0}.$

\item[4)] For all $t\in\lbrack0,T[,~\overrightarrow{\nabla}.(u(t))=0$ in
$S^{\prime}(\mathbb{R}^{d}).$

\item[5)] There exists $p\in D^{\prime}(Q_{T})$ such that $\partial
_{t}u-\Delta u+\nabla(u\otimes u)+\overrightarrow{\nabla}p=0$ in $D^{\prime
}(Q_{T}).$
\end{enumerate}
\end{definition}

In 1934, J. Leray \cite{Ler} proved that, for any initial data $u_{0}$ in
$L^{2}(\mathbb{R}^{d})$ with divergence free, the Navier-Stokes equations have
at least one weak solution $u$ on $]0,+\infty\lbrack$ which, for every $T>0,$
belongs to the Leray energy space $\mathcal{L}_{T}$ defined by:%
\[
\mathcal{L}_{T}=L^{\infty}([0,T],L^{2}(\mathbb{R}^{d}))\cap L^{2}%
([0,T],H^{1}(\mathbb{R}^{d})).
\]
This leads us to introduce the following notion of Leray weak solutions.

\begin{definition}
Let $T>0$ and $u_{0}$ in $L^{2}(\mathbb{R}^{d})$ with divergence free. We call
Leray weak solutions of the equations (NS) on $]0,T[$ every weak solution on
$]0,T[$ of (NS) which belongs to the Leray energy space $\mathcal{L}_{T}.$
\end{definition}

Naturally, the question on the uniqueness of the Leray weak solutions raises.
In the bi-dimension case corresponding to $d=2$, it is well known that such
solutions are unique (see for instance \cite{Tem}). \ However, in the case
$d\geq3,$ the question remains open. We only have some partial answers. In
fact, the uniqueness is obtained under some variant of supplementary
conditions on the regularity of the solutions. As examples, we cite the works
of J. Serrin \cite{Ser}, W. Von Wahl \cite{Von}, J. Y. Chemin \cite{Che}, I.
Gallagher and F. Plonchon \cite{GP}, and P. Germain \cite{Ger}. In This
direction, Q. Chen, C. Miao, and Z. Zhang have recently proved the following
uniqueness result.

\begin{theorem}
[See \cite{CMZ}, Theorem 1.4]Let $T>0$ and $u_{0}$ in
$L^{2}(\mathbb{R}^{d})$ with divergence free. Let $u_{1}$ and $u_{2}$ be two
Leray weak solutions of the equations (NS) on $]0,T[.$ Assume that
\[
u_{1}\in L^{\frac{2}{1-r_{1}}}([0,T],B_{\infty}^{-r_{1},\infty}(\mathbb{R}%
^{d}))\text{ and }u_{2}\in L^{\frac{2}{1-r_{2}}}([0,T],B_{\infty}%
^{-r_{2},\infty}(\mathbb{R}^{d}))
\]
where $0\leq r_{1},r_{2}<1$ and $r_{1}+r_{2}<1.$ Then $u_{1}=u_{2}.$
\end{theorem}

As a consequence, the spaces $L^{\frac{2}{1-r}}([0,T],B_{\infty}^{-r,\infty
}(\mathbb{R}^{d})),$ with $r\in\lbrack0,\frac{1}{2}[,$ constitute a uniqueness
class of Leray weak solutions of (NS). In this paper, we extend this
uniqueness criteria to $r\in\lbrack\frac{1}{2},1]$; which gives a positive
answer to the question of Q. Chen, C. Miao, and Z. Zhang in [\cite{CMZ},
Remark 1.7].

Before setting our results, let us introduce the following notation.

\begin{notation}
Let $T>0$ and $r\in]0,1].$ We denote by $\mathcal{P}_{r,T}$ the space
$L^{\frac{2}{1-r}}([0,T],B_{\infty}^{-r,\infty}(\mathbb{R}^{d}))$ if $r\neq1$
and the space $C([0,T],B_{\infty}^{-1,\infty}(\mathbb{R}^{d}))$ if $r=1.$
\end{notation}

Now we are in position to cite our main results.

\begin{theorem}
\label{Th1}We suppose here that $d\leq4.$ Let $T>0$ and $u_{0}$ in
$L^{2}(\mathbb{R}^{d})$ with divergence free. If $u_{1}$ and $u_{2}$ are two
Leray weak solutions of the equations (NS) on $]0,T[$ such that $u_{1}%
\in\mathcal{P}_{r_{1},T}$ and $u_{2}\in\mathcal{P}_{r_{2},T}$ for some
$r_{1},r_{2}\in]0,1]$ then $u_{1}=u_{2}.$
\end{theorem}

Thanks to the precise Sobolev inequalities proved by P.Gerard, Y. Meyer, and
F. Oru \cite{GMO}, the proof of the former theorem will be a consequence of
the following more general uniqueness result.

\begin{theorem}
\label{Th2}Let $T>0$, $(r_{1},r_{2})\in]0,1]^{2},$ and $(p_{i},q_{i}%
)_{i=1,2}\in\mathbb{R}^{2}$ such that, for each $i,q_{i}\geq d$ and $p_{i}%
\geq\frac{4}{1+r_{i}}$ if $r_{i}\neq1$ and $p_{i}>2$ if $r_{i}=1.$ If $u_{1}$
and $u_{2}$ are two Leray weak solutions on $]0,T[$ of the equations (NS)
associated to the same initial data $u_{0}$ such that, for $i\in\{1,2\},$
\[
u_{i}\in L^{p_{i}}([0,T],L^{q_{i}}(\mathbb{R}^{d}))\cap\mathcal{P}_{r_{i},T},
\]
then $u_{1}=u_{2}.$
\end{theorem}

The proof of this theorem repose essentially on the following regularity result.

\begin{theorem}
\label{Th3}Let $T>0$, $q\geq d,r\in]0,1],$ and $p\geq\frac{4}{1+r}$ such that
$p>2$ if $r=1.$ If $u\in L^{p}([0,T],L^{q}(\mathbb{R}^{d}))\cap\mathcal{P}%
_{r,T}$ is a weak solution of (NS) on $]0,T[,$ then $\sqrt{t}u\in L^{\infty
}([0,T],L^{\infty}(\mathbb{R}^{d}))$ and $\sqrt{t}\left\Vert u(t)\right\Vert
_{\infty}$ tends to $0$ as $t\rightarrow0.$
\end{theorem}

\begin{remark}
\label{Rk11}This theorem implies, in particular, that every weak solution $u$ of the
equations of Navier-Stokes which belongs to the space $L^{p}([0,T],L^{q}%
(\mathbb{R}^{d}))\cap\mathcal{P}_{r,T}$ is a classical solutions of (NS) i.e.
$u\in C^{\infty}(Q_{T})$ (see the proof of \ this result in the last section
of this paper).
\end{remark}

\begin{remark}
\label{Rk12}In the case where $r=1$, the theorem \ref{Th3} has been recently proved by P.
G. Lemarie-Rieusset \cite{Lem07} when $q>d$ and by the author of this paper
\cite{May09} when $q=d.$ Therefore, we will prove the theorem \ref{Th3} only
in the case $r\in]0,1[.$
\end{remark}

\begin{remark}
P. Germain \cite{Ger} proved the uniqueness of Leray weak solutions of (NS) in
the class $L^{\frac{2}{1-r}}([0,T],X_{r})$ with $r\in\lbrack-1,1[$ and
\[
X_{r}=\left\{
\begin{array}
[c]{c}%
M~(H^{r},L^{2}),\text{ if }r\in]0,1[,\\
\Lambda^{r}BMO,\text{if }r\in]-1,0],\\
Lip,\text{ if }r=-1,
\end{array}
\right.
\]
where $Lip=\{f:\mathbb{R}^{d}\rightarrow\mathbb{R}:\left\Vert f\right\Vert
_{Lip}\equiv\sup_{x\neq y}\frac{\left\vert f(x)-f(y)\right\vert }{\left\Vert
x-y\right\Vert }<\infty\},$ $\Lambda^{r}=(I-\Delta)^{\frac{r}{2}}$ and
$M~(H^{r},L^{2})$ is the space of functions $f\in L_{loc}^{2}(\mathbb{R}^{d})$
such that, for every $g\in H^{r}(\mathbb{R}^{d}),$ $fg\in L^{2}(\mathbb{R}%
^{d}).$ The space $M~(H^{r},L^{2})$ is endowed with the norm
\[
\left\Vert f\right\Vert _{M~(H^{r},L^{2})}=\sup_{\left\Vert g\right\Vert
_{H^{r}(\mathbb{R}^{d})}\leq1}\left\Vert fg\right\Vert _{L^{2}(\mathbb{R}%
^{d})}.
\]
Since $X_{r}\hookrightarrow B_{\infty}^{-r,\infty},$ Theorem \ref{Th1} of the
present paper combined with Theorem 1.2 and Theorem 1.4 in \cite{CMZ} extend
the uniqueness class of P. Germain.
\end{remark}

\begin{remark}
H. Miura \cite{Miu} proved the uniqueness of the weak solutions of the
equations (NS) which belongs to the space%
\[
\mathcal{M}_{T}=L^{2}([0,T],L_{uloc}^{2})\cap C([0,T],vmo^{-1})\cap
L_{loc}^{\infty}(]0,T],L^{\infty}),
\]
where $vmo^{-1}$ is the space of $f\in S^{\prime}(\mathbb{R}^{d})$ satisfying%
\[
\forall T>0,~\left\Vert f\right\Vert _{BMO_{T}^{-1}}\equiv\sup_{x_{0}\in
R^{d},~0<R^{2}<T}R^{-\frac{d}{2}}\left(  \int_{[0,R^{2}]\times B(x_{0}%
,R)}\left\vert e^{t\Delta}f\right\vert ^{2}dtdy\right)  ^{\frac{1}{2}}<+\infty
\]
and%
\[
\lim_{T\rightarrow0}\left\Vert f\right\Vert _{BMO_{T}^{-1}}=0.
\]
The space $vmo^{-1}$ is endowed with the norm $\left\Vert .\right\Vert
_{BMO_{T}^{-1}}$ where $T$ is a fixed non negative real number.
\end{remark}

Theorem \ref{Th1} extends this uniqueness result; in fact since $vmo^{-1}%
\hookrightarrow B_{\infty}^{-1,\infty}$ then every weak solution of the
Navier-Stokes equations in the space%
\[
\mathcal{M}_{T}^{p,q}\equiv L^{2}([0,T],L_{uloc}^{2})\cap C([0,T],vmo^{-1}%
)\cap L_{loc}^{p}(]0,T,L^{q}),
\]

with $p>2$ and $q\geq d,$ belongs to Miura's space $\mathcal{M}_{T}.$
Therefore, the family of spaces $(\mathcal{M}_{T}^{p,q})_{p>2,~q\geq d}$
constitute a uniqueness class of weak solutions of Navier-Stokes equations.

The remaining of this paper is organized as follows: in the next section, we
first recall the notion of mild solutions of (NS) introduced in \cite{FLT},
then we cite some useful properties of the Besov spaces and the Chemin-Lerner
spaces. In the third section, we prove how the main theorem \ref{Th3} implies
Theorem \ref{Th2} and Theorem \ref{Th1}. The last section is devoted to the
proof of Theorem \ref{Th3} in the case $r\in]0,1[.$

\section{Preliminairies}

\subsection{Notations}

(1) In this paper, all the functional spaces are defined on the whole space
$\mathbb{R}^{d}.$ Then, in order to simplify the notations, we will design,
for instance, the spaces $L^{q}(\mathbb{R}^{d}),~H^{s}(\mathbb{R}^{d}),$ and
$B_{q}^{s,p}(\mathbb{R}^{d})$ respectively by $L^{q},~H^{s},$ and $B_{q}%
^{s,p}.$

(2) If $X$ is a vector space and $n\in\mathbb{N},$ we often write
$(f_{1},\cdots,f_{n})\in X$ in place of $(f_{1},\cdots,f_{n})\in X^{n}.$

(3) If $X$ is a Banach space, $T>0,$ and $p\in\lbrack1,+\infty],$ we denote by
$L_{T}^{p}(X)$ or $L_{T}^{p}X$ the space $L^{p}([0,T],X).$

(4) Let $p\geq1.$ We design by $\mathbf{E}_{p}$ the space of functions $f\in
L_{loc}^{p}(\mathbb{R}^{d})$ such that
\[
\left\Vert f\right\Vert _{\mathbf{E}_{p}}\equiv\sup_{x_{0}\in\mathbb{R}^{d}%
}\left\Vert 1_{B(x_{0},1)}f\right\Vert _{p}<\infty\text{ and }\lim_{\left\Vert
x_{0}\right\Vert \rightarrow\infty}\left\Vert 1_{B(x_{0},1)}f\right\Vert
_{p}=0.
\]

(5) If $A$ and $B$ are two real valued functions, the notation $A\lesssim B$
means that there exists an absolute non negative real constant $\alpha$ such
that $A\leq\alpha B.$

\subsection{Mild solutions of the Navier-Stokes equations}

We denote by $\mathbb{P}$ the Leray projector on the space of distributions
with divergence free. We recall that $(\mathbb{P}_{ij})_{1\leq i,j\leq d}$ is
defined via Riesz transformations $(\mathcal{R}_{i})_{1\leq i\leq d}$ by the relation:%

\[
\mathbb{P}_{ij}(f)=\delta_{ij}f-\mathcal{R}_{i}\mathcal{R}_{j}(f)
\]
where $\delta_{ij}$ is the Kronecker symbol.

Let $u_{0}=(u_{01},\cdots,u_{0d})\in S^{\prime}(\mathbb{R}^{d})$ a tempered
distribution with divergence free. By applying formally the Leray operator
$\mathbb{P}$ to the equations (NS) we obtain the following system:%

\[
\left\{
\begin{array}
[c]{c}%
u_{t}-\Delta u=-\mathbb{P}\nabla(u\otimes u),\\
u(0,.)=u_{0}(.).
\end{array}
\right.
\]
Next, using Duhamel formula we transform this system to the integral equations%

\begin{equation}
u(t)=e^{t\Delta}u_{0}+\mathbf{B}(u,u)(t), \tag{NSI}%
\end{equation}
where $(e^{t\Delta})_{t\geq0}$ is the heat semi-group and $\mathbf{B}$ is the
bilinear application defined by:%

\[
\mathbf{B}(u,v)=\mathbb{L}_{Oss}(u\otimes v).
\]
The operator $\mathbb{L}_{Oss}$ , called the Oseen integral operator, is given
by%
\begin{equation}
\mathbb{L}_{Oss}(f)(t)=-\int_{0}^{t}e^{(t-s)\Delta}\mathbb{P}\nabla(f)ds.
\label{21}%
\end{equation}
In \cite{FLT}, G. Furioli, P. G. Lemarie-Rieusset, and F. Terraneo proved that
for the solutions class $L_{loc}^{2}([0,T[,\mathbf{E}_{2})$ the equations (NS0
and (NSI) are equivalents. This leads us to introduce the following notion of
mild solutions which we adopt in this paper.

\begin{definition}
Let $T>0$ and $u_{0}\in S^{\prime}(\mathbb{R}^{d}).$ A mild solution of the
Navier-Stokes equations on $]0,T[$ is a function $u\in L_{loc}^{2}%
([0,T[,\mathbf{E}_{2})$ which satisfies, for every $t\in\lbrack0,T[,$ the
integral equations (NSI).
\end{definition}

\begin{remark}
\label{Rk21}It is well-known (see for instance \cite{FLT} and \cite{Lem02}) that every
mild solution $u$ of the equations (NS) on the interval $]0,T[$ belongs to
the space $C([0,T[,B_{\infty}^{-d-1,\infty}).$
\end{remark}

\begin{remark}
All solutions of the Navier-Stokes considered in this paper are mild
solutions; therefore, in the sequel, if $u\in L_{loc}^{2}([0,T[,\mathbf{E}%
_{2}),$ then the short sentence "$u$ is a solution of the equations (NS)"
means that $u$ is a mild solution on $]0,T[$ of the equations (NS).
\end{remark}

\begin{remark}
\label{Rk23}Let $u$ be a mild solution $]0,T[$ of the equations (NS). Using the semi-group
property of $(e^{t\Delta})_{t\geq0}$, one can easily verify that for every
$0<t_{0}\leq t<T$%
\[
u(t)=e^{(t-t_{0})\Delta}u(t_{0})-\int_{t_{0}}^{t}e^{(t-s)\Delta}%
\mathbb{P}\nabla(u\otimes u)ds.
\]
This implies that the function $u_{t_{0}}\equiv u(.+t_{0})$ is a mild solution
on $]0,T-t_{0}[$ of the Navier-Stokes equations associated to the initial data
$u(t_{0}).$
\end{remark}

\begin{remark}
In the sequel of this work, the hypothesis of free divergence of the solutions
$u$ of the Navier-Stokes equations (NS) will play no role.
\end{remark}

\subsection{Besov spaces and Chemin-Lerner spaces}

let us first recall the Littlewood-Paley decomposition. Let $\varphi\in
C_{c}^{\infty}(\mathbb{R}^{d})$ which is equal to $1$ on a neighbourhood of
the origin. Next we define the function $\psi\in C_{c}^{\infty}(\mathbb{R}%
^{d}\backslash\{0\})$ by $\psi(\xi)=\varphi(\frac{\xi}{2})-\varphi(\xi).$ For
every $j\in\mathbb{N}\cup\{0\},$ we design by $S_{j}$ and $\Delta_{j}$ the
operators defined on $S^{\prime}(\mathbb{R}^{d})$ and $S^{\prime}%
(\mathbb{R}\times\mathbb{R}^{d})$ by
\begin{align*}
S_{j}f  &  =\mathcal{F}_{x}^{-1}(\varphi(\frac{\xi}{2^{j}})\mathcal{F}%
_{x}(f)),\\
\Delta_{j}f  &  =\mathcal{F}_{x}^{-1}(\psi(\frac{\xi}{2^{j}})\mathcal{F}%
_{x}(f)),
\end{align*}
where $\mathcal{F}_{x}$ and $\mathcal{F}_{x}^{-1}$ are respectively the
Fourier transformation with respect to the space variable $x\in\mathbb{R}^{d}$
and its inverse transformation.

\begin{notation}
In the sequel, we often denote the operator $S_{0}$ by $\Delta_{-1}.$
\end{notation}

Now we can recall the definition of a class of Besov spaces.

\begin{definition}
Let $s\in\mathbb{R}$ and $q\in\lbrack1,+\infty].$ The Besov space
$B_{q}^{s,\infty}$ is the space of $f\in S^{\prime}(\mathbb{R}^{d})$ such that%
\[
\left\Vert f\right\Vert _{B_{q}^{s,\infty}}\equiv\sup_{j\geq-1}2^{si}%
\left\Vert \Delta_{j}f\right\Vert _{q}<\infty.
\]
We design by $\tilde{B}_{q}^{s,\infty}$ the closure of $S(\mathbb{R}^{d})$ in
$B_{q}^{s,\infty}.$
\end{definition}

We introduce now the definition of a class of Chemin-Lerner spaces
[\cite{Dan}, \cite{Che}, \cite{CL}].

\begin{definition}
Let $T>0,~s\in\mathbb{R}$ and $p,q\in\lbrack1,+\infty].$ The Chemin-Lerner
space $\tilde{L}_{T}^{p}B_{q}^{s,\infty}$ is the space of $v\in S^{\prime
}(\mathbb{R}\times\mathbb{R}^{d})$ such that%
\[
\left\Vert v\right\Vert _{\tilde{L}_{T}^{p}B_{q}^{s,\infty}}\equiv\sup
_{j\geq-1}2^{si}\left\Vert \Delta_{j}v\right\Vert _{L_{T}^{p}L_{x}^{q}}%
<\infty.
\]
We design by $\mathbf{\tilde{L}}_{T}^{p}B_{q}^{s,\infty}$ the space of
$v\in\tilde{L}_{T}^{p}B_{q}^{s,\infty}$ such that%
\[
\lim_{T\rightarrow0}\left\Vert v\right\Vert _{\tilde{L}_{T}^{p}B_{q}%
^{s,\infty}}=0.
\]

\end{definition}

The following proposition gathers some simple and useful properties of Besov
and Chemin-Lerner spaces.

\begin{proposition}
\label{Pr21}Let $T>0,s\in\mathbb{R}$, $(p,q)\in\lbrack1,+\infty],$ and $p_{1}\in
\lbrack1,+\infty\lbrack.$ The following assertions hold true:
\begin{enumerate}
\item[1)] $L_{T}^{p}B_{q}^{s,\infty}\hookrightarrow\tilde{L}_{T}^{p}%
B_{q}^{s,\infty},~L_{T}^{\infty}B_{q}^{s,\infty}=\tilde{L}_{T}^{\infty}%
B_{q}^{s,\infty},$ and $L_{T}^{p_{1}}B_{q}^{s,\infty}\hookrightarrow
\mathbf{\tilde{L}}_{T}^{p_{1}}B_{q}^{s,\infty}.$

\item[2)] The linear operators $\mathbb{P}_{ij}\frac{\partial}{\partial
x_{k}}$ are continuous from $B_{q}^{s,\infty}$ (respectively $L_{T}^{p}%
B_{q}^{s,\infty}$) to $B_{q}^{s-1,\infty}$ (respectively $L_{T}^{p}%
B_{q}^{s-1,\infty}$).

\item[3)] (Bernestein's inequality) For every $m\in\lbrack q,+\infty],$ we
have%
\[
B_{q}^{s,\infty}\hookrightarrow B_{q}^{s+d(\frac{1}{m}-\frac{1}{q}),\infty
}\text{ and }\tilde{L}_{T}^{p}B_{q}^{s,\infty}\hookrightarrow\tilde{L}_{T}%
^{p}B_{q}^{s+d(\frac{1}{m}-\frac{1}{q}),\infty}.
\]
\end{enumerate}
\end{proposition}

The proof of this proposition is classical and simple.\smallskip

It is well-known (see for instance \cite{Can}, \cite{Lem02}, \cite{Tri}) that
Besov spaces can be characterized via the heat semi group $(e^{t\Delta
})_{t\geq0}.$ The following proposition is a particular case of such characterization.

\begin{proposition}
\label{Pr22}Let $q\in\lbrack1,+\infty]$ and $s>0$. Then for each real number $\delta>0,$
the quantity%
\[
\sup_{0<\theta<\delta}\theta^{\frac{s}{2}}\left\Vert e^{s\Delta}f\right\Vert
_{q}%
\]
defines a norm on the Besov space $B_{q}^{-s,\infty}$ equivalent to the
original norm $\left\Vert .\right\Vert _{B_{q}^{-s,\infty}}.$
\end{proposition}

In order to study the properties of the pointwise product we introduce the
following modified and simplified version of the Bony para-product. For every
$f$ and $g$ in $S^{\prime}(\mathbb{R}^{d})$ (or in $S^{\prime}(\mathbb{R}%
\times\mathbb{R}^{d}))$, we define formally $\Pi_{1}(f,g)$ and $\Pi_{2}(f,g)$
by
\[
\Pi_{1}(f,g)=\sum_{j=-1}^{\infty}S_{j+1}f\Delta_{j}g\text{ and }\Pi
_{2}(f,g)=\sum_{j=0}^{\infty}S_{j}f\Delta_{j}g.
\]
So we have, at least formally, the equality $fg=\Pi_{1}(f,g)+\Pi_{2}(g,f).$
The operators $\Pi_{1}$ and $\Pi_{2}$ will be called \ "\textit{the operators
of the Bony para-product}".

The next proposition describes some continuity properties of the Bony
para-product operators on Besov spaces and Chemin-Lerner spaces.

\begin{proposition}
\label{Pr23}Let $T>0,~\sigma_{2}>\sigma_{1}>0$ two non negative reals numbers, and
$(p_{1},q_{1}),(p_{2},q_{2})\in\lbrack1,+\infty]^{2}$ such that $\frac{1}%
{p}\equiv\frac{1}{p_{1}}+\frac{1}{p_{2}}\leq1$ and $\frac{1}{q}\equiv\frac
{1}{q_{1}}+\frac{1}{q_{2}}\leq1.$ Then the following assertions hold true:

\begin{enumerate}
\item[1)] The operators $\Pi_{1}$ and $\Pi_{2}$ are continuous from
$B_{q_{1}}^{-\sigma_{1},\infty}\times B_{q_{2}}^{\sigma_{2},\infty}$ to
$B_{q}^{\sigma_{2}-\sigma_{1},\infty}.$

\item[2)] The operators $\Pi_{1}$ and $\Pi_{2}$ are continuous from
$\tilde{L}_{T}^{p_{1}}B_{q_{1}}^{-\sigma_{1},\infty}\times\tilde{L}_{T}%
^{p_{2}}B_{q_{2}}^{\sigma_{2},\infty}$ to $\tilde{L}_{T}^{p}B_{q}^{\sigma
_{2}-\sigma_{1},\infty}$ and from $L_{T}^{p_{1}}L_{x}^{q_{1}}\times\tilde
{L}_{T}^{p_{2}}B_{q_{2}}^{\sigma_{2},\infty}$ to $\tilde{L}_{T}^{p}%
B_{q}^{\sigma_{2},\infty}.$ Moreover their norms are independent of $T.$
\end{enumerate}
\end{proposition}

The proof of this proposition is simple, see for instance \cite{Che} and
\cite{Lem02} where similar results are proved.

We study now the regular effect of the heat equations measured in term of
Besov spaces and Chemin-Lerner spaces.

The first result concerns the regular effect of the semi-group $(e^{t\Delta
})_{t\geq0}.$

\begin{proposition}
\label{Pr24}[Regular effect of the heat semi-group]Let $T>0,(s_{1},s_{2},s_{3}%
)\in\mathbb{R}^{3},$ and $(p,q)\in\lbrack1,+\infty]^{2}.$ Then we have the
following assertions:

\begin{enumerate}
\item[1)] If $s_{1}\leq s_{2}$ then the family $(t^{\frac{s_{2}-s_{1}}{2}%
}e^{t\Delta})_{0<t\leq T}$ is bounded in the space $\mathcal{L}(B_{q}%
^{s_{1},\infty},B_{q}^{s_{2},\infty}).$

\item[2)] The operator $e^{t\Delta}$ is continuous from $B_{q}^{s,\infty}$ to
$\tilde{L}_{T}^{p}B_{q}^{s+\frac{2}{p},\infty}.$ Moreover, if $p<\infty$ then
$e^{t\Delta}$ is continuous from $B_{q}^{s,\infty}$ to $\mathbf{\tilde{L}}%
_{T}^{p}B_{q}^{s+\frac{2}{p},\infty}.$
\end{enumerate}
\end{proposition}

The second result concerns the regular effect of the integral Oseen operator
$\mathbb{L}_{Oss}$ defined by (\ref{21}).

\begin{proposition}
\label{Pr25}Let $T>0,s\in\mathbb{R},$ and $(p_{1},p_{2},q)\in\lbrack1,+\infty]^{3}$ such
that $p_{1}\leq p_{2}$ and set $s^{\prime}\equiv s+1-2(\frac{1}{p_{1}}%
-\frac{1}{p_{2}}).$Then the Oseen operator $\mathbb{L}_{Oss}$ maps boundly the
space $\tilde{L}_{T}^{p_{1}}B_{q}^{s,\infty}$ into $\tilde{L}_{T}^{p2}%
B_{q}^{s^{\prime},\infty}$ and its norm is majorized by $C(1+T)$ where $C$ is
a non negative constant independent of $T.$
\end{proposition}

For the proofs of these two propositions, we refer the reader to \cite{Che}
and \cite{Dan}.

\section{Proof of Theorems \ref{Th1} and \ref{Th2}}

In this short section, we will see how Theorem \ref{Th3}, which will be proved
in the next section, allows to prove Theorem \ref{Th2} and how Theorem
\ref{Th2} \ implies Theorem \ref{Th1}.

\subsection{Proof of Theorem \ref{Th2}}

First, from Theorem \ref{Th3}, we have for $i=1$ or $2$%

\[
u_{i}\in L_{T}^{p}\mathbf{E}_{d},~\sqrt{t}u_{i}\in L_{T}^{\infty}L_{x}%
^{\infty}\text{, and }\lim_{t\rightarrow0}\sqrt{t}\left\Vert u_{i}%
(t)\right\Vert _{\infty}=0,
\]
where $p=\inf(p_{1},p_{2})$; (see the section 2.1 for the definition of the
space $\mathbf{E}_{d}).$ Set $u\equiv u_{1}-u_{2};$ this function satisfies
the equation%
\[
u=\mathbf{B}(u_{1},u)+\mathbf{B}(u,u_{2}).
\]
Using the continuity on the space $\mathbf{E}_{d}$ of the pointwise
multiplication with a function in $L^{\infty}(\mathbb{R}^{d})$ and the
convolution with a function in $L^{1}(\mathbb{R}^{d})$ and recalling that
$e^{(t-s)\Delta}\mathbb{P}\nabla$ is a convolution operator and that the
$L^{1}(\mathbb{R}^{d})$ norm of its kernel does not exceed $C(t-s)^{-\frac
{1}{2}}$ for some absolute constant $C>0$ (see for instance \cite{Can},\cite{Lem02}, and \cite{Mey}), we easily deduce that, for every $t\in\lbrack0,T],$ we have%
\[
\left\Vert u(t)\right\Vert _{\mathbf{E}_{d}}\leq C~\omega(t)\int_{0}^{t}%
\frac{\left\Vert u(s)\right\Vert _{\mathbf{E}_{d}}}{\sqrt{t-s}\sqrt{s}}ds,
\]
where%
\[
\omega(t)\equiv\sup_{0<s\leq t}\sqrt{s}\left(  \left\Vert u_{1}(s)\right\Vert
_{\mathbf{E}_{d}}+\left\Vert u_{2}(s)\right\Vert _{\mathbf{E}_{d}}\right)  .
\]
Invoking now the continuity of the linear operator%
\[
L(f)(t)\equiv\int_{0}^{t}\frac{f(s)}{\sqrt{t-s}\sqrt{s}}ds
\]
on the space $L^{p}(\mathbb{R}^{+})$ (see for instance Lemma 5.2
\cite{May09}), we deduce that for every $\delta\in]0,T]$ we have%
\[
\sup_{0<t\leq\delta}\left\Vert u(t)\right\Vert _{\mathbf{E}_{d}}\leq
C_{p}~\omega(\delta)\sup_{0<t\leq\delta}\left\Vert u(t)\right\Vert
_{\mathbf{E}_{d}},
\]
where $C_{p}$ is a constant which only depends on $p.$ Hence, by using the
fact that $\omega(\delta)\rightarrow0$ as $\delta\rightarrow0$ we infer that
there exists $\delta\in]0,T]$ such that $u=0$ (and by consequent $u_{1}%
=u_{2})$ on $[0,\delta].$ Finally, thanks to a classical iteration argument
(see for example Lemma 27.2 in \cite{Lem07}) we conclude that $u_{1}=u_{2}$ on
$[0,T].$

\subsection{Proof of Theorem \ref{Th1}}

Let $i=1$ or $2.$ First let us notice that the classical interpolation in the
Lebesgue spaces and Sobolev spaces implies that $u_{i}$ belongs to the space
$L_{T}^{\frac{2}{r_{i}}}H^{r_{i}}.$ Using now the nonhomogene version of the
precise Sobolev inequalities proved by P. Gerard, Y. Meyer, and F. Oru
\cite{GMO} (see also \cite{Lem07} of an other proof)%
\begin{align*}
\left\Vert f\right\Vert _{q}  &  \lesssim\left(  \left\Vert f\right\Vert
_{W^{\alpha,p}}\right)  ^{1-\frac{\alpha}{\beta}}\left(  \left\Vert
f\right\Vert _{B_{\infty}^{\alpha-\beta,\infty}}\right)  ^{\frac{\alpha}%
{\beta}},\\
0  &  <\alpha<\beta,~1<p<\infty,~\frac{p}{q}=(1-\frac{\alpha}{\beta}),
\end{align*}
with $\alpha=r_{i},\beta=2r_{i}$ and $p=2,$ we get the inequality%
\[
\left\Vert u_{i}(t)\right\Vert _{4}\lesssim\left(  \left\Vert u(t)\right\Vert
_{H^{r_{i}}}\right)  ^{\frac{1}{2}}\left(  \left\Vert u(t)\right\Vert
_{B_{\infty}^{-r_{i},\infty}}\right)  ^{\frac{1}{2}}%
\]
which, thanks to Holder inequality, implies that
\[
u_{i}\in L^{4}([0,T],L^{4}(\mathbb{R}^{d})).
\]
Hence, applying Theorem \ref{Th2} completes the proof of Theorem \ref{Th1}.

\section{Proof of the theorem \ref{Th3}}

This section is devoted to the proof of the main theorem \ref{Th3} in the case
where $r\in]0,1[$ (see Remark \ref{Rk12}). The proof repose on some
intermediate results.

The first proposition is a local uniqueness result under a supplementary
regularity hypothesis on the initial data $u_{0}.$

\begin{proposition}
\label{Pr41}Let $T>0,r\in]0,1[,q\geq d,$ and $p\geq\frac{4}{1+r}.$ If the
initial data $u_{0}\in L^{q}(\mathbb{R}^{d})$ and $u1,u_{2}\in L_{T}^{\frac
{2}{1-r}}(B_{\infty}^{-r,\infty})\cap L_{T}^{p}L_{x}^{q}$ are two mild
solutions on $]0,T[$ of the equations (NS), then there exists $\delta\in]0,T[$
such that $u1=u_{2}$ on $[0,\delta].$
\end{proposition}

In the second proposition, we prove a result of regularity persistency and a
criterion of eventual finite time explosion of regular solutions of
Navier-Stokes equations.

\begin{proposition}
\label{Pr42}Let $q\geq d$ a real number and $u_{0}\in L^{q}(\mathbb{R}^{d}).$
Then the following assertions hold:

\begin{enumerate}
\item[1)] The Navier-Stokes equations (NS) havent a unique maximal solution
$u$ in the space $C([0,T^{\ast}[,L^{q}(\mathbb{R}^{d})).$ Moreover, for every $\sigma>0,$
$u\in C^{\infty}(]0,T^{\ast}[,\tilde{B}_{\infty}^{\sigma,\infty}).$

\item[2)] If in addition $u_{0}\in B_{\infty}^{-r,\infty}$ for some
$r\in]0,1[$ then the maximal solution $u$ belongs to the space $L_{loc}%
^{\infty}([0,T^{\ast}[,B_{\infty}^{-r,\infty}).$

\item[3)] If $T^{\ast}<\infty$ then for every $r\in]0,1[$ there exists a
constant $\varepsilon_{r,d}>0$, which depends only on $r$ and $d,$ such that
\begin{equation}
\underline{\lim}_{t\rightarrow T^{\ast}}(T^{\ast}-t)^{\frac{1-r}{2}}\left\Vert
u(t)\right\Vert _{B_{\infty}^{-r,\infty}}\geq\varepsilon_{r,d}. \label{42}%
\end{equation}
In particular,%
\[
\int_{T^{\ast}/2}^{T^{\ast}}\left\Vert u(t)\right\Vert _{B_{\infty}%
^{-r,\infty}}dt=+\infty.
\]

\end{enumerate}
\end{proposition}

\begin{remark}
Estimation (\ref{42}) improve a similar result of Y. Giga \cite{Gig} where the
Besov space $B_{\infty}^{-r,\infty}$ is replaced by the Lebesgue space
$L^{\frac{d}{r}}(\mathbb{R}^{d}).$ (Recall that $L^{\frac{d}{r}}%
(\mathbb{R}^{d})\hookrightarrow B_{\infty}^{-r,\infty}).$
\end{remark}

The last preliminary result concerns the behavior as $t\rightarrow0$ of the
regular solutions $u(t)$ of the equations (NS) which belong to the class
$L_{T}^{\frac{2}{1-r}}(B_{\infty}^{-r,\infty}).$

\begin{proposition}
\label{Pr43} Let $r\in]0,1[,T>0,$ and $u\in C(]0,T],B_{\infty}^{1,\infty})\cap
L_{T}^{\frac{2}{1-r}}(B_{\infty}^{-r,\infty})$ a mild solution on $]0,T[$ of
the equations of Navier-Stokes. Then $\sqrt{t}\left\Vert u(t)\right\Vert
_{\infty}\rightarrow0$ as $t\rightarrow0.$
\end{proposition}

Let us now see how the above propositions allow together to prove the main
theorem \ref{Th3}.

\begin{proof}
[Proof of Theorem \ref{Th3}]Set $\Omega_{q,r}\equiv\{t_{0}\in]0,T]:~u(t_{0}%
)\in L^{q}(\mathbb{R}^{d})\cap B_{\infty}^{-r,\infty}\}.$ let $t_{0}$ be an
arbitrary element of $\Omega_{q,r}.$ According to Proposition \ref{Pr42}, the
equations (NS) with initial data $u(t_{0})$ have a unique maximal solution
$v\in C([0,T^{\ast}[,L^{q}(\mathbb{R}^{d}))\cap L_{loc}^{\infty}([0,T^{\ast
}[,B_{\infty}^{-r,\infty}).$ Hence Remark \ref{Rk21} and Proposition
\ref{Pr41} insure the existence of $\delta\in]0,\delta_{0}\equiv\min(T^{\ast
},T-t_{0})[$ such that $v=u(.+t_{0})$ on $[0,\delta].$ This allow to define
\[
\delta_{\ast}\equiv\sup\{\delta\in]0,\delta_{0}[:v=u(.+t_{0})\text{ on
}[0,\delta]\}.
\]
Suppose that $\delta_{\ast}<\delta_{0};$ then the facts that $v$ is in
 $C([0,\delta_{0}[,L^{q}(\mathbb{R}^{d}))$ and $u(.+t_{0})$ belongs to $C([0,\delta
_{0}[,B_{\infty}^{-d-1,\infty})$ (see Remark \ref{Rk21}) imply that
$v(\delta_{\ast})=u(\delta_{\ast}+t_{0})\in L^{q}(\mathbb{R}^{d}).$ Hence, by
applying another time the proposition \ref{Pr41} to the Navier-stokes
equations with initial data $v(\delta_{\ast}),$ we deduce the existence of
$\delta^{\prime}>\delta_{\ast}$ such that $v=u(.+t_{0})$ on $[0,\delta
^{\prime}].$ This contradicts the definition of $\delta_{\ast};$ we then infer
that $v=u(.+t_{0})$ on $[0,\delta_{0}[.$ Therefore since by assumption $u\in
L_{T}^{\frac{2}{1-r}}(B_{\infty}^{-r,\infty}),$ we get $v\in L^{\frac{2}{1-r}%
}([0,\delta_{0}[,B_{\infty}^{-r,\infty})$ which implies, thanks to the last
assertion of \ref{Pr42}, that $u(.+t_{0})=v$ on $[0,T-t_{0}].$ Using now the
regularity property of $v$ ensured by the first assertion of \ref{Pr42} and
the fact that $\Omega_{q,r}$ is dense in $]0,T],$ we conclude that the
solution $u$ belongs to the space $\cap_{\sigma>0}C^{\infty}(]0,T],B_{\infty
}^{\sigma,\infty}).$ Finally, Proposition \ref{Pr43} ends the proof of Theorem
\ref{Th3}.
\end{proof}

\subsection{Proof of Proposition \ref{Pr41}}

In order to prove this proposition, we will fellow an approach inspired by the
paper \cite{Che}\ of J.Y. Chemin. We will decompose our proof into two steps.

\subsubsection{The first step}

Let $u_{0}\in L^{q}(\mathbb{R}^{d})$ and $u\in L_{T}^{\frac{2}{1-r}}%
(B_{\infty}^{-r,\infty})\cap L_{T}^{p}L_{x}^{q}$ a solution of the
Navier-Stokes equations with initial data $u_{0}.$ We will prove that there
exists $T_{0}\in]0,T]$ such that $u\in\mathbf{\tilde{L}}_{T_{0}}^{\frac
{2}{1+r}}(B_{q}^{1+r,\infty}).$ To do this, we will need the following useful lemmas.

\begin{lemma}
\label{L41}Let $\delta\in]0,T],\rho\in\lbrack\frac{2}{1+r},+\infty\lbrack,$
$m\in\lbrack1,+\infty],$ and $\sigma\in]r,+\infty\lbrack.$ Then the linear
operator $\mathbb{L}_{u},$ defined by:%
\begin{equation}
\mathbb{L}_{u}(f)=\sum_{k=1}^{2}\mathbb{L}_{oss}(\Pi_{k}(u,f)), \label{43}%
\end{equation}
is bounded on the space $\mathbf{\tilde{L}}_{\delta}^{\rho}(B_{m}%
^{\sigma,\infty})$ and its norm is less than $C\left\Vert u\right\Vert
_{L_{\delta}^{\frac{2}{1-r}}(B_{\infty}^{-r,\infty})}$ where $C$ is an
absolute non negative constant independent of $\delta.$
\end{lemma}

\begin{lemma}
\label{L42}Set $\omega=\mathbf{B}(u,u)$ and $\omega_{0}=\mathbb{L}%
_{u}(e^{t\Delta}u_{0})$ where $\mathbb{L}_{u}$ is the operator defined by
(\ref{43}). Then we have

\begin{enumerate}
\item[1)] $\omega\in\mathbf{\tilde{L}}_{T}^{\frac{p}{2}}(B_{\frac{q}{2}%
}^{1,\infty}).$

\item[2)] $\omega_{0}\in\mathbf{\tilde{L}}_{T}^{\frac{2}{1+r}}(B_{q}%
^{1+r,\infty}).$

\item[3)] $\omega_{0}\in\mathbf{\tilde{L}}_{T}^{\frac{p}{2}}(B_{\frac{q}{2}%
}^{1+\frac{2}{p},\infty}).$
\end{enumerate}
\end{lemma}

\begin{lemma}
\label{L43}Let $X_{1}$ and $X_{2}$ be two Banach space and let $f$ be a
function defined on $X_{1}$ and $X_{2}$ such that $f:X_{1}\rightarrow X_{1}$
and $f:X_{2}\rightarrow X_{2}$ are contractions. Then the fixed point of
$f$ in $X_{1}$ belongs to $X_{2}.$
\end{lemma}

Let us admit for a moment these lemmas and prove that there exists $T_{0}%
\in]0,T]$ such that $u\in\mathbf{\tilde{L}}_{T_{0}}^{\frac{2}{1+r}}%
(B_{q}^{1+r,\infty}).$

\begin{proof}
Set $\omega=\mathbf{B}(u,u)$ and $\omega_{0}=\mathbb{L}_{u}(e^{t\Delta}u_{0})$
as in Lemma \ref{L42}, and consider the following decomposition of $\omega:$%
\[
\omega=\omega_{0}+\mathbb{L}_{u}(\omega)\equiv F_{u}(\omega).
\]
Lemma \ref{L41} and the last two assertions of Lemma \ref{L42} ensure that,
for $T_{0}\in]0,T]$ small enough such that $\left\Vert u\right\Vert _{L_{T_{0}%
}^{\frac{2}{1-r}}(B_{\infty}^{-r,\infty})}$ be less than an absolute constant
depending only on $r,p,$ and $q,$ the function $F_{u}$ is a contraction on the
Banach spaces $\mathbf{\tilde{L}}_{T_{0}}^{\frac{2}{1+r}}(B_{q}^{1+r,\infty})$
and $\mathbf{\tilde{L}}_{T_{0}}^{\frac{p}{2}}(B_{\frac{q}{2}}^{1+\frac{2}%
{p},\infty}).$ But the first assertion of Lemma \ref{L42} and the definition
of the application $F_{u}$ imply that $\omega$ is the fixed point of $F_{u}$
in the space $\mathbf{\tilde{L}}_{T_{0}}^{\frac{p}{2}}(B_{\frac{q}{2}%
}^{1+\frac{2}{p},\infty});$ hence lemma \ref{L43} yields $\omega
\in\mathbf{\tilde{L}}_{T_{0}}^{\frac{2}{1+r}}(B_{q}^{1+r,\infty}).$ Therefore,
thanks to the first assertion of Lemma \ref{L42} and the fact that
$u=\omega_{0}+\omega$, we conclude that $u$ belongs to the space
$\mathbf{\tilde{L}}_{T_{0}}^{\frac{2}{1+r}}(B_{q}^{1+r,\infty}).$
\end{proof}

Let us now prove the above three lemmas.

\begin{proof}
[Proof of Lemma \ref{L41}]It is a direct consequence of the injection
$L_{T}^{\frac{2}{1-r}}(B_{\infty}^{-r,\infty})\hookrightarrow\mathbf{\tilde
{L}}_{T}^{\frac{2}{1-r}}(B_{\infty}^{-r,\infty}),$ the continuity of the Bony
para-product operators $\Pi_{k}$ on the Chemin-Lerner spaces (see
Proposition \ref{Pr23}), and the regularizing effect of the Ossen integral
operator $\mathbb{L}_{Oss}$ (see Proposition \ref{Pr25}).
\end{proof}

\smallskip

\begin{proof}
[Proof of Lemma \ref{L42}]From Holder inequality, $u\otimes u\in L_{T}%
^{\frac{p}{2}}L_{x}^{\frac{q}{2}}.$ Hence the injection $L_{T}^{\frac{p}{2}%
}L_{x}^{\frac{q}{2}}$ in $\mathbf{\tilde{L}}_{T}^{\frac{p}{2}}(B_{\frac{q}{2}%
}^{0,\infty})$ and Proposition \ref{Pr25} imply that $\omega\in\mathbf{\tilde
{L}}_{T}^{\frac{p}{2}}(B_{\frac{q}{2}}^{1,\infty}).$ To prove the second and
the third assertion, we first notice in view of the injection of
$L^{q}(\mathbb{R}^{d})$ in $\tilde{B}_{q}^{0,\infty}$ and Proposition
\ref{Pr24} we have $e^{t\Delta}u_{0}$ belongs to the space $\mathbf{\tilde{L}%
}_{T}^{\frac{2}{1+r}}(B_{q}^{1+r,\infty})$ and the space $\mathbf{\tilde{L}%
}_{T}^{p}(B_{q}^{\frac{2}{p},\infty}).$ Therefore, Lemma \ref{L41} implies
that $\omega_{0}\in\mathbf{\tilde{L}}_{T}^{\frac{2}{1+r}}(B_{q}^{1+r,\infty
}).$ Finally, the continuity of the Bony para-product operators
\[
\Pi_{k}:L_{T}^{p}L_{x}^{q}\times\mathbf{\tilde{L}}_{T}^{p}(B_{q}^{\frac{2}%
{p},\infty})\rightarrow\mathbf{\tilde{L}}_{T}^{\frac{p}{2}}(B_{\frac{q}{2}%
}^{\frac{2}{p},\infty})
\]
and the regularizing effect of the Ossen integral operator $\mathbb{L}_{Oss}$
(see Proposition \ref{Pr25}) yield that $\omega_{0}\in\mathbf{\tilde{L}}%
_{T}^{\frac{p}{2}}(B_{\frac{q}{2}}^{1+\frac{2}{p},\infty}).$
\end{proof}

\smallskip

\begin{proof}
[Proof of Lemma \ref{L43}]We consider the Banach space $X=X_{1}\cap X_{2}$
endowed with the norm $\left\Vert .\right\Vert =\left\Vert .\right\Vert
_{X_{1}}+\left\Vert .\right\Vert _{X2}.$ It is clear that $f$ is a contraction
on $X,$ hence in view of the Banach's fixed point theorem it has a unique
fixed point $z^{\prime}$ in $X$ which, thanks to the fact that $X\subset
X_{1}$ and the Banach's fixed point theorem, is the unique fixed point of $f$
in the space $X_{1}.$
\end{proof}

\subsubsection{The second step}

Let $u_{1},u_{2}\in L_{T}^{\frac{2}{1-r}}(B_{\infty}^{-r,\infty})\cap
L_{T}^{p}L_{x}^{q}$ be two solutions of the Navier-Stokes equations with the
same initial data $u_{0}\in L^{q}(\mathbb{R}^{d}).$ According to above step,
$u_{1},u_{2}\in\mathcal{Z}_{T_{0}}\equiv\mathbf{\tilde{L}}_{T_{0}}^{\frac
{2}{1+r}}(B_{q}^{1+r,\infty})\cap\mathbf{\tilde{L}}_{T_{0}}^{\frac{2}{1-r}%
}(B_{\infty}^{1-r,\infty})$ for some $T_{0}\in]0,T].$ Let $\delta\in]0,T_{0}]$
to be fixed later. A simple application of Proposition \ref{Pr23} and
Proposition \ref{Pr25} implies that the Bilinear operator $\mathbf{B}$ is
continuous from $\mathcal{Z}_{\delta}\times\mathcal{Z}_{\delta}$ to
$\mathbf{\tilde{L}}_{\delta}^{\frac{2}{1+r}}(B_{q}^{1+r,\infty})\cap
\mathbf{\tilde{L}}_{\delta}^{\frac{2}{1-r}}(B_{q}^{1-r,\infty})$ and its norm
is bounded by a constant $C$ independent of $\delta.$ Using now the
Berstein's injection $\mathbf{\tilde{L}}_{\delta}^{\frac{2}{1-r}}%
(B_{q}^{1-r,\infty})\hookrightarrow\mathbf{\tilde{L}}_{\delta}^{\frac{2}{1-r}%
}(B_{\infty}^{1-r-\frac{d}{q},\infty})$ and the assumption $q\geq d,$ we
deduce that $\mathbf{B}:\mathcal{Z}_{\delta}\times\mathcal{Z}_{\delta
}\rightarrow\mathcal{Z}_{\delta}$ is continuous and therefore%
\[
\left\Vert u_{1}-u_{2}\right\Vert _{\mathcal{Z}_{\delta}}\leq C(\left\Vert
u_{1}\right\Vert _{\mathcal{Z}_{\delta}}+\left\Vert u_{2}\right\Vert
_{\mathcal{Z}_{\delta}})\left\Vert u_{1}-u_{2}\right\Vert _{\mathcal{Z}%
_{\delta}},
\]
with $C$ independent of $\delta.$ Hence, up to choose $\delta$ small enough,
we conclude that $u_{1}=u_{2}$ on $[0,\delta].$

\subsection{Proof of Proposition \ref{Pr42}}

The proof of the first assertion of the proposition \ref{Pr42} is classical
and well-known (see for instance \cite{Can}, \cite{Lem02}, \cite{Mey}). The
prove of the assertions (2) and (3) repose essentially on the following
elementary lemma where the following notation is used.

\begin{notation}
Let $T,\mu>0$ be two non negative real numbers. We denote by $L_{\mu
,T}^{\infty}$ the space of measurable functions $f:[0,T[\times\mathbb{R}%
^{d}\rightarrow\mathbb{R}^{d}$ such that
\[
\left\Vert f\right\Vert _{L_{\mu,T}^{\infty}}\equiv\sup_{0<s<T}s^{\frac{\mu
}{2}}\left\Vert f(s)\right\Vert _{\infty}<\infty.
\]

\end{notation}

\begin{lemma}
\label{L44}Let $r\in]0,1[$ and $T>0.$ Then the bilinear operator $B$ is
continuous from $L_{1,T}^{\infty}\times L_{r,T}^{\infty}$ (respectively
$L_{r,T}^{\infty}\times L_{r,T}^{\infty}$) into $L_{r,T}^{\infty}$ with norm
less than $C_{r,d}$ (respectively $C_{r,d}T^{\frac{1-r}{2}}$) where $C_{r,d}$
is a non negative constant which depends only on $r$ and $d.$
\end{lemma}

The proof of this lemma is simple, we just have to recall that $e^{(t-s)\Delta
}\mathbb{P}\nabla$ is a convolution operator with an integrable function with
$L^{1}(\mathbb{R}^{d})$ norm of order  $\frac{1}{\sqrt{t-s}}$.

Now we are ready to prove the last two assertions of the proposition
\ref{Pr42}.

\begin{proof}
It is well-known (see for example \cite{Lem02}) that there exists $T_{0}%
\in]0,T]$ such that the solution $u$ given in the first assertion is the limit
in the Banach space $X_{T_{0}}\equiv L_{1,T_{0}}^{\infty}\cap L_{r,T_{0}%
}^{\infty}$ of the sequence $(u_{(n)})_{n}$ defined by:%
\begin{align*}
u_{(0)}  &  =e^{t\Delta}u_{0},\\
\forall n  &  \in\mathbb{N},~u_{(n+1)}=u_{(0)}+\mathbf{B}(u_{(n)},u_{(n)}),
\end{align*}
and $(\sigma_{n}\equiv\left\Vert u_{(n+1)}-u_{(n)}\right\Vert _{X_{T_{0}}})\in
l^{1}(\mathbb{N}).$ We will prove that $(u_{(n)})_{n}$ is a cauchy sequence in
the space $L_{r,T_{0}}^{\infty}.$ First, notice that since $u_{0}\in
B_{\infty}^{-r,\infty}$ then, from Proposition \ref{Pr22}, $u_{(0)}\in
L_{r,T_{0}}^{\infty}.$ Hence, by iteration, Lemma \ref{L44} guarantees that
the sequence $(u_{(n)})_{n}$ belongs to the space $L_{r,T_{0}}^{\infty}$ and
satisfies the following inequality:%
\[
\left\Vert u_{(n+1)}-u_{(n)}\right\Vert _{L_{r,T_{0}}^{\infty}}\leq
C_{r,d}\sigma_{n}\left(  \left\Vert u_{(n)}\right\Vert _{L_{r,T_{0}}^{\infty}%
}+\left\Vert u_{(n-1)}\right\Vert _{L_{r,T_{0}}^{\infty}}\right)  ,
\]
which implies (see \cite{FLZZ}) that $(u_{(n)})_{n}$ is a Cauchy sequence
$L_{r,T_{0}}^{\infty}.$ Using now the fact that $L_{r,T_{0}}^{\infty}$ is a
complete space, that $L_{r,T_{0}}^{\infty}\hookrightarrow L_{1,T_{0}}^{\infty
},$ and the fact that $(u_{(n)})_{n}$ converges to the solution $u$ in
$L_{1,T_{0}}^{\infty},$ we deduce that $u\in L_{r,T_{0}}^{\infty}.$ Let us now
show that $u\in L_{T_{0}}^{\infty}(B_{\infty}^{-r,\infty})$ which ends the
proof of the second assertion of our proposition. First, since $u_{0}\in
B_{\infty}^{-r,\infty}$ then, from Proposition \ref{Pr24}, $e^{t\Delta}%
u_{0}\in L_{T_{0}}^{\infty}(B_{\infty}^{-r,\infty}).$ Second, Proposition
\ref{Pr22} and Young's inequalities imply that for every $t\in]0,T_{0}],$%
\begin{align*}
\left\Vert \mathbf{B}(u,u)(t)\right\Vert _{B_{\infty}^{-r,\infty}}  &
\lesssim\sup_{0<\theta\leq1}\theta^{\frac{r}{2}}\left\Vert e^{\theta\Delta
}\mathbf{B}(u,u)(t)\right\Vert \\
&  \lesssim\sup_{0<\theta\leq1}\theta^{\frac{r}{2}}\int_{0}^{t}\frac
{1}{s^{\frac{1+r}{2}}\sqrt{t+\theta-s}}ds\left\Vert u\right\Vert _{L_{r,T_{0}%
}^{\infty}}\left\Vert u\right\Vert _{L_{1,T_{0}}^{\infty}}\\
&  \lesssim\left\Vert u\right\Vert _{L_{r,T_{0}}^{\infty}}\left\Vert
u\right\Vert _{L_{1,T_{0}}^{\infty}}.
\end{align*}
This completes the proof since $u=e^{t\Delta}u_{0}+\mathbf{B}(u,u).$

Let us prove the last assertion of Proposition \ref{Pr42}. Suppose that
$T^{\ast}<\infty.$ Let $r\in]0,1[$ and $t_{0}$ in $I_{\ast}\equiv
]\max(0,T^{\ast}-),T^{\ast}[.$ According to the remark \ref{Rk23}, $u_{t_{0}%
}\equiv u(.+t_{0})$ satisfies on $[0,\delta_{0}\equiv T^{\ast}-t_{0}[$ the
following equality%
\[
u_{t_{0}}(t)=e^{t\Delta}u(t_{0})+\mathbf{B}(u_{t_{0}},u_{t_{0}})(t).
\]
Hence, in view of Proposition \ref{Pr22} and Lemma \ref{L44}, we have for
every $t\in\lbrack0,\delta_{0}[$%
\[
t^{\frac{r}{2}}\left\Vert u(t)\right\Vert _{\infty}\leq C\left(  \left\Vert
u(t_{0})\right\Vert _{B_{\infty}^{-r,\infty}}+t^{\frac{1-r}{2}}\left(
\sup_{0<s\leq t}t^{\frac{r}{2}}\left\Vert u(s)\right\Vert _{\infty}\right)
^{2}\right)  .
\]
Define $f(t)\equiv\sup_{0<s\leq t}t^{s\frac{r}{2}}\left\Vert u(s)\right\Vert
_{\infty}.$ It yields that for any $t\in\lbrack0,\delta_{0}[,$
\[
f(t)\leq C\left(  \left\Vert u(t_{0})\right\Vert _{B_{\infty}^{-r,\infty}%
}+(T^{\ast}-t_{0})^{\frac{1-r}{2}}f^{2}(t)\right)  .
\]
Recalling now that $\left\Vert u(t)\right\Vert _{\infty}\rightarrow+\infty$ as
$t\rightarrow T^{\ast}$ (see \cite{Gig}, \cite{Lem02}, \cite{May03}), which is
equivalent to $f(t)\rightarrow+\infty$ as $t\rightarrow T^{\ast}$; we deduce
from the elementary lemma below that
\[
\left\Vert u(t_{0})\right\Vert _{B_{\infty}^{-r,\infty}}(T^{\ast}%
-t_{0})^{\frac{1-r}{2}}\geq\varepsilon_{r,d}\equiv\frac{1}{4C^{2}},
\]
which ends the proof of our proposition.
\end{proof}

\begin{lemma}
Let $a<b$ two real numbers and $f:[a,b[\rightarrow R$ a continuous function.
Assume that there exist two reals numbers $A,B>0$ such that $4AB<1,~f(0)\leq
2A,$ and, for every $t\in\lbrack a,b[,~f(t)\leq A+Bf^{2}(t).$ Then, for every
$t\in\lbrack a,b[,~f(t)\leq2A.$
\end{lemma}

The proof of this lemma is simple, it suffices to apply the intermediate value
theorem after noticing that if $4AB<1$ then $f(t)\neq2A$ for every
$t\in\lbrack a,b[.$

\subsection{The proof of Proposition \ref{Pr43}}

The proof of this paper is inspired by the paper \cite{Lem07} of P.G.
Lemarie'-Rieusset. Let $(t_{n})_{n}\in]0,\frac{T}{2}[$ such that
$t_{n}\rightarrow0$ as $n\rightarrow+\infty.$ We consider the sequence of
functions $(u_{n})_{n}$ definite on $[0,\frac{T}{2}]$ by $u_{n}(t)=u(t_{n}+t).$
We have to prove that $\sup_{0<t<\delta}\sqrt{t}\left\Vert u_{n}(t)\right\Vert
_{\infty}$ converges to $0$ as $\delta$ goes to $0$ uniformly on $n.$ Firstly,
in order to simplify the writing, we introduce the following notations:%
\begin{align*}
h_{n}(\mu,\delta)  &  \equiv\sup_{0<t<\delta}t^{\frac{\mu+1}{2}}\left\Vert
u_{n}(t)\right\Vert _{B_{\infty}^{\mu,\infty}},\\
\Theta(\delta)  &  =\sup_{0<t<\frac{T}{2}}\left\Vert u\right\Vert
_{L^{\frac{2}{1-r}}([t_{0},t_{0}+\delta],B_{\infty}^{-r,\infty})}.
\end{align*}
Let $\sigma\in]r,1[$ a fixed real number and $\delta_{0}\in]0,\frac{T}{2}[$ to
be chosen later. Let $n\in\mathbb{N},\delta\in]0,\delta_{0}[,$ and
$t\in]0,\delta].$ Let $a=a(n,t)$ be an element of the interval $[\frac{t}%
{4},\frac{t}{2}]$ such that%
\[
\left\Vert u_{n}(a)\right\Vert _{B_{\infty}^{-r,\infty}}=\inf_{s\in
\lbrack\frac{t}{4},\frac{t}{2}]}\left\Vert u_{n}(s)\right\Vert _{B_{\infty
}^{-r,\infty}}.
\]
Since $u_{n}$ is a mild solution of the equations of Navier-Stokes then,
according to Remark \ref{Rk23},%
\begin{align}
u_{n}(t)  &  =e^{(t-a)\Delta}u_{n}(a)-\int_{a}^{t}e^{(t-s)}\mathbb{P}%
\nabla(u_{n}\otimes_{n})ds\label{44}\\
&  \equiv I_{n}(t)+J_{n}(t). \label{45}%
\end{align}
Now we will estimate the norm of the terms $I_{n}(t)$ and $J_{n}(t)$ in the
Besov space $B_{\infty}^{\sigma,\infty}.$ According to the first assertion of
Proposition \ref{Pr24} and the definition of $a=a(n,t)$ we have%
\begin{align}
\left\Vert I_{n}(t)\right\Vert _{B_{\infty}^{\sigma,\infty}}  &  \lesssim
t^{-\frac{\sigma+r}{2}}\left\Vert u_{n}(a)\right\Vert _{B_{\infty}^{-r,\infty
}}\nonumber\\
&  \lesssim t^{-\frac{1+\sigma}{2}}\left\Vert u_{n}\right\Vert _{L^{\frac
{2}{1-r}}([\frac{t}{4},\frac{t}{2}],B_{\infty}^{-r,\infty})}\nonumber\\
&  \lesssim t^{-\frac{1+\sigma}{2}}\Theta(\delta). \label{46}%
\end{align}
On the other hand since $J_{n}(t)=\mathbb{L}_{Oss}(1_{[a,t]}u_{n}%
\otimes1_{[a,t]}u_{n}),$ then by using the continuity of the Bony para-product
operators $\Pi_{k}$ from $\tilde{L}_{T}^{\frac{2}{1-r}}(B_{\infty}^{-r,\infty
})\times L_{T}^{\infty}(B_{\infty}^{\sigma,\infty})$ to $\tilde{L}_{T}%
^{\frac{2}{1-r}}(B_{\infty}^{\sigma-r,\infty})$ (see Proposition \ref{Pr23})
and the continuity of the operator $\mathbb{L}_{Oss}$ from $\tilde{L}%
_{T}^{\frac{2}{1-r}}(B_{\infty}^{\sigma-r,\infty})$ to $L_{T}^{\infty
}(B_{\infty}^{\sigma,\infty})$ (see Proposition \ref{Pr25}) we easily deduce
that%
\begin{align}
\left\Vert J_{n}(t)\right\Vert _{B_{\infty}^{\sigma,\infty}}  &
\lesssim\left\Vert u_{n}\right\Vert _{\tilde{L}^{\frac{2}{1-r}}%
([a,t],B_{\infty}^{-r,\infty})}\left\Vert u_{n}\right\Vert _{L^{\infty
}([a,t],B_{\infty}^{\sigma,\infty})}\nonumber\\
&  \lesssim\left\Vert u_{n}\right\Vert _{L^{\frac{2}{1-r}}([a,t],B_{\infty
}^{-r,\infty})}\sup_{a\leq s\leq t}\left\Vert u_{n}(s)\right\Vert _{B_{\infty
}^{\sigma,\infty}}\nonumber\\
&  \lesssim t^{-\frac{1+\sigma}{2}}\Theta(\delta)h_{n}(\sigma,\delta
)\nonumber\\
&  \lesssim t^{-\frac{1+\sigma}{2}}\Theta(\delta_{0})h_{n}(\sigma,\delta).
\label{47}%
\end{align}
estimates (\ref{46}) and (\ref{47}) imply that there exists a constant
$C_{1}>0$ independent of $t,\delta,$ and $n$ such that%
\[
h_{n}(\sigma,\delta)\leq C_{1}\Theta(\delta)+C_{1}\Theta(\delta_{0}%
)h_{n}(\sigma,\delta).
\]
Hence, by choosing $\delta_{0}$ small enough such that $\Theta(\delta_{0}%
)\leq\frac{1}{2C_{1}}$ (which is always possible since $\Theta(\delta
_{0})\rightarrow0$ as $\delta_{0}\rightarrow0),$ we get%
\begin{equation}
h_{n}(\sigma,\delta)\leq C_{1}\Theta(\delta). \label{49}%
\end{equation}
Now let us go back to (\ref{44}) and (\ref{45}) in order to estimate the norms
of $I_{n}(t)$ and $J_{n}(t)$ in the Besov space $B_{\infty}^{-r,\infty}.$
Firstly, in view of Proposition \ref{Pr24} and the definition of $a=a(n,t)$ we
have%
\begin{align}
\left\Vert I_{n}(t)\right\Vert _{B_{\infty}^{-r,\infty}}  &  \lesssim
\left\Vert u_{n}(a)\right\Vert _{B_{\infty}^{-r,\infty}}\nonumber\\
&  \lesssim t^{-\frac{1-r}{2}}\Theta(\delta). \label{410}%
\end{align}
Secondly, by using the continuity of the operators $\Pi_{k}$ from $B_{\infty
}^{-r,\infty}\times B_{\infty}^{\sigma,\infty}$ to $B_{\infty}^{\sigma
-r,\infty},$ the action of $\mathbb{P}\nabla$ on Besov spaces (Proposition
\ref{Pr21}), and the first assertion of Proposition \ref{Pr24}, we get%
\begin{align}
\left\Vert J_{n}(t)\right\Vert _{B_{\infty}^{-r,\infty}}  &  \lesssim\int
_{a}^{t}\frac{1}{(t-s)^{\frac{1-\sigma}{2}}}\left\Vert \mathbb{P}\nabla
(u_{n}\otimes u_{n})\right\Vert _{B_{\infty}^{\sigma-r-1,\infty}}ds\nonumber\\
&  \lesssim t^{\frac{1+\sigma}{2}}\sup_{\frac{t}{4}<s<t}\left\Vert
u_{n}(s)\right\Vert _{B_{\infty}^{-r,\infty}}\sup_{\frac{t}{4}<s<t}\left\Vert
u_{n}(s)\right\Vert _{B_{\infty}^{\sigma,\infty}}\nonumber\\
&  \lesssim t^{-\frac{1-r}{2}}h_{n}(-r,\delta)h_{n}(\sigma,\delta)\nonumber\\
&  \lesssim t^{-\frac{1-r}{2}}h_{n}(-r,\delta)\Theta(\delta_{0}), \label{411}%
\end{align}
where we have used (\ref{49}) in the last inequality.

Combining (\ref{410}) and (\ref{411}), we deduce the existence of an absolute
constant $C_{2}>0$ independent of $t,\delta,$ and $n$ such that%
\[
h_{n}(-r,\delta)\leq C_{2}\Theta(\delta)+C_{2}\Theta(\delta_{0})h_{n}%
(-r,\delta).
\]
Hence, for $\delta_{0}$ small enough, we have%
\begin{equation}
h_{n}(-r,\delta)\leq2C_{2}\Theta(\delta). \label{412}%
\end{equation}
Using now the following interpolation inequality \cite{LM}%
\[
\left\Vert f\right\Vert _{\infty}\leq\left(  \left\Vert f\right\Vert
_{B_{\infty}^{-r,\infty}}\right)  ^{\frac{\sigma}{r+\sigma}}\left(  \left\Vert
f\right\Vert _{B_{\infty}^{\sigma,\infty}}\right)  ^{\frac{r}{r+\sigma}},
\]
we deduce from (\ref{49}) and (\ref{412}) that there exist two constants $C>0$
and $\delta_{0}\in]0,\frac{T}{2}]$ independent of $n$ such that for every
$\delta\in]0,\delta_{0}]$ we have%
\[
\sup_{0<t<\delta}\sqrt{t}\left\Vert u_{n}(t)\right\Vert _{\infty}\leq
C\Theta(\delta),
\]
which completes the proof of Proposition \ref{Pr43}.

\end{document}